\pdfoutput=1

\documentclass[11pt]{amsart}

\usepackage{epigamath}

\usepackage[english]{babel}

\numberwithin{equation}{section}

\usepackage{enumitem}

\usepackage{amsthm}
\usepackage{amsmath}
\usepackage{amssymb}
\usepackage{amscd}
\usepackage{mathtools}
\usepackage[all]{xy}
\usepackage{indentfirst}
\usepackage{comment}
\usepackage{microtype}
\usepackage{tikz}
\usepackage{tikz-cd}
\usetikzlibrary{patterns}
\usepackage{enumitem}

\usepackage{framed}
\usepackage{color}
\definecolor{shadecolor}{gray}{0.875}
\definecolor{col}{RGB}{42, 95, 151}

\newtheorem{thm}{\bf Theorem}[section]

\newtheorem{lem}[thm]{\bf Lemma}

\newtheorem{cor}[thm]{\bf Corollary}

\newtheorem{q}[thm]{\bf Question}

\newtheorem*{thm*}{\bf Theorem}
\newtheorem*{cor*}{\bf Corollary}

\theoremstyle{definition}

\newtheorem{rem}[thm]{\it Remark}

\newtheorem*{df*}{\bf Definition}
\newtheorem*{not*}{\bf Notation}
\newtheorem*{dfs*}{\bf Definitions}
\newtheorem*{ack*}{\bf Acknowledgements}
\newtheorem*{dfrem*}{\bf Definition and Remark}
\newtheorem*{notc*}{\bf Notation and Convention}

\def\P{\mathbb{P}}
\def\C{\mathbb{C}}

\def\Z{\mathbb{Z}}

\DeclareMathOperator{\Supp}{Supp}

\DeclareMathOperator{\Ker}{Ker}

\DeclareMathOperator{\Coker}{Coker}

\DeclareMathOperator{\alg}{alg}

\DeclareMathOperator{\Ext}{Ext}
\DeclareMathOperator{\MHS}{MHS}
\DeclareMathOperator{\rat}{rat}

\theoremstyle{definition}

\EpigaVolumeYear{5}{2021}
\EpigaArticleNr{20}
\ReceivedOn{December 11, 2020}
\InFinalFormOn{November 10, 2021}
\AcceptedOn{November 23, 2021}

\title{Factorization of the Abel--Jacobi maps}
\titlemark{Factorization of the Abel--Jacobi maps}

\author{Fumiaki Suzuki}
\address{UCLA Mathematics Department, Box 951555, Los Angeles, CA, 90095-1555}
\email{suzuki@math.ucla.edu}

\authormark{Fumiaki Suzuki}

\AbstractInEnglish{As an application of the theory of Lawson homology and morphic cohomology,
Walker proved that the Abel--Jacobi map factors through another regular homomorphism.
In this note, we give a direct proof of the theorem.}



\TitleInFrench{Factorisation des applications d'Abel--Jacobi}

\AbstractInFrench{Walker a montr\'e que l'application d'Abel--Jacobi se factorise \`a travers un autre homomorphisme r\'egulier,
en utilisant la th\'eorie de l'homologie de Lawson et la cohomologie morphique. Dans cet article, nous donnons une
preuve directe de ce th\'eor\`eme.}

\begin{document}



\maketitle

\begin{prelims}

\DisplayAbstractInEnglish

\bigskip

\DisplayKeyWords

\medskip

\DisplayMSCclass

\bigskip

\languagesection{Fran\c{c}ais}

\bigskip

\DisplayTitleInFrench

\medskip

\DisplayAbstractInFrench

\end{prelims}


\newpage

\setcounter{tocdepth}{2}

\tableofcontents


\section{Introduction}
For a smooth complex projective variety $X$, 
the Abel--Jacobi map $AJ^{p}$ provides a fundamental tool to study codimension $p$ cycles on $X$.
It is a homomorphism of abelian groups
\[
AJ^{p}\colon CH^{p}(X)_{\hom}\rightarrow J^{p}(X),
\]
where
$CH^{p}(X)_{\hom}$ is the group of codimension $p$ cycles homologous to zero modulo rational equivalence
and
\[
J^{p}(X)= H^{2p-1}(X, \C)/(H^{2p-1}(X,\Z(p)) + F^{p}H^{2p-1}(X, \C))
\] 
is the $p^\mathrm{th}$ Griffiths intermediate Jacobian.
Although the Abel--Jacobi map $AJ^{p}$ is transcendental by nature,
it is well-known that we get an algebraic theory by restricting it to the subgroup $A^{p}(X)\subset CH^{p}(X)_{\hom}$ of cycles algebraically equivalent to zero.
More precisely,
the image $J^{p}_{a}(X)\subset J^{p}(X)$ of $A^{p}(X)$ under $AJ^{p}$ is an abelian variety,
and the induced map
\[
\psi^{p}\colon A^{p}(X)\rightarrow J^{p}_{a}(X),
\]
which we also call Abel--Jacobi, 
satisfies the following 
\cite{Gr, Lie}:
for any smooth connected 
projective 
variety $S$ with a base point $s_{0}$ and for any codimension $p$ cycle $\Gamma$ on $S\times X$,
the composition 
\[
\left\{\begin{array}{l}S\rightarrow A^{p}(X) \rightarrow J^{p}_{a}(X)\\
s\mapsto \psi^{p}(\Gamma_{s}-\Gamma_{s_{0}})
\end{array}\right.
\]
is a morphism of algebraic varieties.

Generally, for a given abelian variety $A$, a homomorphism $\phi\colon A^{p}(X)\rightarrow A$ with the analogous property is called {\it regular} (this definition goes back to the work of Samuel \cite{Sa}). Remarkably, the Abel--Jacobi map $\psi^{p}$ factors through another regular homomorphism due to a theorem of Walker \cite{W},
which was originally proved as an application of the theory of Lawson homology and morphic cohomology. 
The purpose of this note is to give a direct proof of the theorem.

\begin{thm}[\protect{\cite[Corollary~5.9]{W}}]\label{UAJw}
For a smooth projective variety $X$,
the Abel--Jacobi map $\psi^{p}$ factors as
\[
\xymatrix{
 &J\left(N^{p-1}H^{2p-1}(X, \Z(p))\right)\ar[d]^-{\pi^{p}}\\
 A^{p}(X)\ar[ur]^-{\widetilde{\psi}^{p}}\ar[r]_-{\psi^{p}}& J^{p}_{a}(X),\\
}
\]
where $J\left(N^{p-1}H^{2p-1}(X, \Z(p))\right)$ is the intermediate Jacobian for the 
pure Hodge structure of weight $-1$
given by the $(p-1)$-stage of the coniveau filtration
\[
N^{p-1}H^{2p-1}(X,\Z(p)) = \Ker\left(H^{2p-1}(X,\Z(p))\rightarrow \varinjlim_{Z\in \mathcal{Z}^{p-1}} H^{2p-1}(X-Z, \Z(p))\right),
\]
 $\pi^{p}$ is a natural isogeny,
and $\widetilde{\psi}^{p}$ is a surjective regular homomorphism.
\end{thm}

\begin{rem}
The Walker map $\widetilde{\psi}^{p}$ is a unique lift of the Abel--Jacobi map $\psi^{p}$.
This follows from the fact that $A^{p}(X)$ is divisible \cite[Lemma 7.10]{BO} and $\Ker(\pi^{p})$ is finite.
\end{rem}

Theorem \ref{UAJw} is related to a classical question of Murre (\emph{cf.}~\cite[Section 7]{M3} and \cite[p. 132]{GMV}), asking whether the Abel--Jacobi map $\psi^{p} \colon A^{p}(X)\rightarrow J^{p}_{a}(X)$ is {\it universal} among all regular homomorphisms $\phi\colon A^{p}(X)\rightarrow A$, that is, whether every such $\phi$ factors through $\psi^{p}$.
This is known to hold for $p=1$ by the theory of the Picard variety, for $p=\dim X$ by the theory of the Albanese variety, and for $p=2$ as proved by Murre \cite{M1,M2} (see \cite{K} for the correction of a gap in the original proof) using the Merkurjev-Suslin theorem \cite{MS}.
Nevertheless, it was recently observed by Ottem--Suzuki \cite{OS} that $\psi^{p}$ is no longer universal in general for $3\leq p\leq \dim X-1$, which settled Murre's question. 
In fact, they constructed a $4$-fold on which the Walker map $\widetilde{\psi}^{3}$ is universal and the isogeny $\pi^{3}$ has non-zero kernel.
We note that the $4$-fold was obtained from a certain pencil of Enriques surfaces with non-algebraic integral Hodge classes of non-torsion type.

Theorem \ref{UAJw} has several other consequences.
It was recently proved by Voisin \cite{V2} that 
the $(n-2)$-stage of the coniveau and strong coniveau filtrations 
\[N^{n-2}H^{2n-3}(X,\Z), \,\widetilde{N}^{n-2}H^{2n-3}(X,\Z)\] always coincides modulo torsion on a rationally connected $n$-fold $X$ 
(see \cite{BeO} for the definition and properties of the strong coniveau filtration).
This result follows from a geometric argument involving families of semi-stable maps from curves to $X$, 
combined with an analogue of the Roitman theorem for the Walker maps $\widetilde{\psi}^{p}$ on a smooth projective variety with small Chow groups \cite[Theorems~1.1 and~A.3]{S}.
The Roitman-type theorem
also allows us to describe the torsion part of the kernel of the Abel--Jacobi maps $\psi^{p}$ in terms of the coniveau under the same assumption on the Chow groups.

Our new proof of Theorem \ref{UAJw} only depends on the Bloch-Ogus theory \cite{BO} and the theory of intermediate Jacobians of mixed Hodge structures.
This simplifies to a large extent the original argument due to Walker, which relies on the full machinery of Lawson homology and morphic cohomology.

We work over the complex numbers throughout.

\subsection*{Acknowledgments}

The author would like to thank Henri Gillet for interesting discussions. He also thanks the anonymous referees for their careful reading and comments.

\section{Proof of the main theorem}

Before beginning the proof, we review the construction of the Abel--Jacobi maps using mixed Hodge structures \cite{J}
(the reader can consult \cite{Del1,Del2} for basic knowledge about mixed Hodge structures). 

For a mixed Hodge structure $(H, W_{\bullet}, F^{\bullet})$, we define its intermediate Jacobian $J(H)$ as the extension group
\[
J(H)=\Ext^{1}_{\MHS}(\Z(0), H)
\]
in the abelian category $\MHS$ of mixed Hodge structures\footnote{We denote by $\Z(m)$ the Hodge structure of Tate $(2\pi i)^{m}\cdot \Z$, which is a pure Hodge structure of weight $-2m$.}.
If $H$ is pure of weight $-1$, then $J(H)$ is isomorphic to a complex torus
\[
H_{\C}/(H_{\Z}+F^{0}H_{\C}).
\]
Let $X$ be a smooth projective variety.
Then the cohomology group $H^{2p-1}(X,\Z(p))$ has a pure Hodge structure of weight $-1$, therefore we have 
\[J^{p}(X)=J\left(H^{2p-1}(X,\Z(p))\right).\]
On the other hand, for a codimension $p$ closed subset $Y\subset X$,
the long exact sequence for cohomology groups with supports gives a short exact sequence\footnote{
For a variety $X$, we denote by $Z^{p}(X)$ the group of codimension $p$ cycles on $X$ 
and by $Z^{p}(X)_{\rat}$ (resp. $Z^{p}(X)_{\alg}$, $Z^{p}(X)_{\hom}$) the subgroup of cycles rationally equivalent to zero (resp. algebraically equivalent to zero, homologous to zero) on $X$.
For a codimension $p$ closed subset $Y\subset X$, we denote by $Z^{p}_{Y}(X)$ the subgroup of cycles supported on $Y$;
the groups $Z^{p}_{Y}(X)_{\rat}, Z^{p}_{Y}(X)_{\alg}$, and $Z^{p}_{Y}(X)_{\hom}$ are accordingly defined.}
\[
0\rightarrow H^{2p-1}(X,\Z(p))\rightarrow H^{2p-1}(X-Y,\Z(p))\rightarrow Z^{p}_{Y}(X)_{\hom}\rightarrow 0.
\]
This is a short exact sequence of mixed Hodge structures, where $Z^{p}_{Y}(X)_{\hom}$ has the trivial Hodge structure.
Then the boundary map in the long exact sequence for $\Ext^{i}_{\MHS}(\Z(0), -)$ determines a map
\[
Z^{p}_{Y}(X)_{\hom}\rightarrow J^{p}(X).
\]
Now we take the direct limit 
over
all codimension $p$ closed subsets of $X$ 
to obtain a map
\[
Z^{p}(X)_{\hom}\rightarrow J^{p}(X).
\]
This coincides with the Abel--Jacobi map $AJ^{p}$ defined by using currents.

\subsection{Construction}
We will use a variant of the above construction to construct the Walker maps.
For a codimension $p$ closed subset $Y\subset X$, the long exact sequence for cohomology groups with supports gives a commutative diagram
\[
\xymatrixcolsep{1pc}
\xymatrix{
0 \ar[r] & H^{2p-1}(X, \Z(p)) \ar[r] \ar[d]^{f} & H^{2p-1}(X- Y, \Z(p)) \ar[r]\ar[d] &Z^{p}_{Y}(X)_{\hom}\ar[r]\ar[d] & 0\\
0 \ar[r] & \varinjlim_{Z\in \mathcal{Z}^{p-1}} H^{2p-1}(X- Z, \Z(p))\ar@{=}[r] &\varinjlim_{Z\in \mathcal{Z}^{p-1}} H^{2p-1}(X- Y- Z, \Z(p))\ar[r]&0\ar[r]&0,\\
}
 \]
 where $\mathcal{Z}^{p-1}$ is the set of codimension $p-1$ closed subsets of $X$.
By the snake lemma, we have an exact sequence
\[
0\rightarrow N^{p-1}H^{2p-1}(X,\Z(p))\rightarrow N^{p-1}H^{2p-1}(X- Y, \Z(p)) \rightarrow Z^{p}_{Y}(X)_{\hom} \xrightarrow{\delta_{Y}} \Coker(f).
\]

We prove that $\Ker(\delta_{Y})=Z^{p}_{Y}(X)_{\alg}$.
We have a commutative diagram with exact rows and columns
\[
\xymatrixcolsep{1pc}
\xymatrix{
& \varinjlim_{(Y,Z)\in \mathcal{Z}^{p}/\mathcal{Z}^{p-1}} H^{2p-1}_{Z-Y}(X- Y, \Z(p)) \ar@{=}[r]\ar[d] & \varinjlim_{(Y,Z)\in \mathcal{Z}^{p}/\mathcal{Z}^{p-1}} H^{2p-1}_{Z-Y}(X-Y, \Z(p))\ar[d]^{\partial}\\
 H^{2p-1}(X, \Z(p)) \ar[r]\ar@{=}[d] & \varinjlim_{Y\in \mathcal{Z}^{p}} H^{2p-1}(X-Y, \Z(p)) \ar[r]\ar[d] & \varinjlim_{Y\in \mathcal{Z}^{p}} H^{2p}_{Y}(X, \Z(p))=Z^{p}(X)\ar[d]\\
 H^{2p-1}(X, \Z(p)) \ar[r]^<<<<<<{f} & \varinjlim_{Z\in \mathcal{Z}^{p-1}}H^{2p-1}(X-Z, \Z(p)) \ar[r]& \varinjlim_{Z\in \mathcal{Z}^{p-1}} H^{2p}_{Z}(X, \Z(p)),\\
}
\]
where 
$\mathcal{Z}^{p}$ is the set of codimension $p$ closed subsets of $X$ and 
$\mathcal{Z}^{p}/\mathcal{Z}^{p-1}$ is the set of pairs $(Y,Z)\in \mathcal{Z}^{p}\times \mathcal{Z}^{p-1}$ such that $Y\subset Z$.
Then the result follows from
 the diagram 
 and the fact that 
the map $\partial$, which can be identified with the differential $E^{p-1,p}_{1}\rightarrow E^{p,p}_{1}$ of the coniveau spectral sequence, 
has the image
$Z^{p}(X)_{\alg}\subset Z^{p}(X)$ \cite[Theorem 7.3]{BO}.

As a consequence, we have a short exact sequence
 \[
 0\rightarrow N^{p-1}H^{2p-1}(X,\Z(p))\rightarrow N^{p-1}H^{2p-1}(X- Y, \Z(p)) \rightarrow Z^{p}_{Y}(X)_{\alg}\rightarrow 0.
 \]
 This is a short exact sequence of mixed Hodge structures, 
 where $N^{p-1}H^{2p-1}(X,\Z(p))$ has a pure Hodge structure of weight $-1$ and $Z^{p}_{Y}(X)_{\alg}$ has the trivial Hodge structure.
Then the boundary map in the long exact sequence for $\Ext_{\MHS}^{i}(\Z(0), -)$ determines a map
 \[
 \widetilde{\psi}^{p}_{Y}\colon Z^{p}_{Y}(X)_{\alg}\rightarrow J\left(N^{p-1}H^{2p-1}(X,\Z(p))\right),
 \]
where $J\left(N^{p-1}H^{2p-1}(X,\Z(p))\right)$ is a complex torus.
Now we take the direct limit 
to obtain a map
 \[
 \widetilde{\psi}^{p}\colon Z^{p}(X)_{\alg}\rightarrow J\left(N^{p-1}H^{2p-1}(X,\Z(p))\right),
 \]
 which we call the Walker map.
 
 \subsection{Basic Properties}
To finish the proof of Theorem \ref{UAJw}, we 
need to establish several basic properties of the Walker map $\widetilde{\psi}^{p}$.
\begin{lem}\label{Al0}
 We have a commutative diagram
 \[
 \xymatrix{
 Z^{p}(X)_{\alg}\ar[r]^<<<<<<{\widetilde{\psi}^{p}}\ar@{^{(}->}[d]& J\left(N^{p-1}H^{2p-1}(X,\Z(p))\right)\ar[d]^{\pi^{p}}\\
 Z^{p}(X)_{\hom}\ar[r]^{AJ^{p}} & J^{p}(X),\\
 }
 \]
  where $\pi^{p}$ is induced by the inclusion $N^{p-1}H^{2p-1}(X,\Z(p))\subseteq H^{2p-1}(X,\Z(p))$.
 \end{lem}
 \begin{proof}
We have a commutative diagram of short exact sequences of mixed Hodge structures
\[
\xymatrixcolsep{1pc}
\xymatrix{
0 \ar[r] & N^{p-1}H^{2p-1}(X, \Z(p)) \ar[r] \ar[d] & N^{p-1}H^{2p-1}(X- Y, \Z(p)) \ar[r]\ar[d] &Z^{p}_{Y}(X)_{\alg}\ar[r]\ar[d] & 0\\
0 \ar[r] & H^{2p-1}(X, \Z(p)) \ar[r] & H^{2p-1}(X- Y, \Z(p)) \ar[r] &Z^{p}_{Y}(X)_{\hom}\ar[r] & 0\\
}
 \]
for any codimension $p$ closed subset $Y\subset X$.
 The assertion follows by applying $\Ext^{i}_{\MHS}(\Z(0),-)$ and taking the direct limit.
\end{proof}
 
 \begin{lem}\label{Al1}
 Let $C$ be a smooth projective curve and $\Gamma$ be a codimension $p$ cycle on $C\times X$ each of whose components dominates $C$.
 Then we have a commutative diagram:
 \[
 \xymatrix{
 Z^{1}(C)_{\hom}\ar[r]^{AJ^{1}}\ar[d]^{\Gamma_{*}}& J^{1}(C)\ar[d]^{\Gamma_{*}}\\
 Z^{p}(X)_{\alg}\ar[r]^<<<<<<{\widetilde{\psi}^{p}} & J\left(N^{p-1}H^{2p-1}(X,\Z(p))\right).\\
 }
 \]
 \end{lem} 
 \begin{proof}
 We freely use the fact that the Betti cohomology and the Borel-Moore homology form a Poincar\'e duality theory with supports (see \cite{Ba,BO} for the axioms).
Let $\pi_{C}\colon C\times X\rightarrow C$ (resp. $\pi_{X}\colon C\times X\rightarrow X$) be the projection to $C$ (resp. $X$).
For a codimension one closed subset $Y\subset C$, setting $Y'=\pi_{C}^{-1}(Y)$,
we have a commutative diagram
\begin{equation}\label{Ad1}
\xymatrixcolsep{1pc}
\xymatrix{
0 \ar[r] & H^{1}(C, \Z(1)) \ar[r] \ar[d]^{(\pi_{C})^{*}} & H^{1}(C- Y, \Z(1)) \ar[r]\ar[d]^{(\pi_{C})^{*}} &Z^{1}_{Y}(C)_{\hom}\ar[r]\ar[d]^{(\pi_{C})^{*}} & 0\\
0 \ar[r] & H^{1}(C\times X, \Z(1)) \ar[r] & H^{1}(C\times X- Y', \Z(1)) \ar[r] &Z^{1}_{Y'}(C\times X)_{\hom}\ar[r] & 0.\\
}
 \end{equation}
Similarly, setting $G=\Supp(\Gamma)$ and $Y''=Y'\cap G$,
we have a commutative diagram
\[
\xymatrixcolsep{1pc}
\xymatrix{
0 \ar[r] & H^{1}(C\times X, \Z(1)) \ar[r] \ar[d]^{\cup\Gamma} & H^{1}(C\times X- Y', \Z(1)) \ar[r]\ar[d]^{(\cup\Gamma)'} &Z^{1}_{Y'}(C\times X)_{\hom}\ar[r]\ar[d]^{\cup \Gamma} & 0\\
0 \ar[r] & H^{2p+1}(C\times X, \Z(p+1)) \ar[r] & H^{2p+1}(C\times X- Y'' , \Z(p+1)) \ar[r] &Z^{p+1}_{Y''}(C\times X)_{\hom}\ar[r] & 0,\\
}
 \]
where, letting $i\colon G-Y''\rightarrow C\times X-Y'$ be a closed immersion and denoting by $H^{BM}_{*}$ the Borel-Moore homology, the middle vertical map $(\cup \Gamma)'$ is the composition
\begin{center}
\begin{tikzcd}[row sep=normal, column sep=small]
H^{1}(C\times X-Y', \Z(1)) \arrow[r,"{i^*}"]
&  H^{1}(G-Y'', \Z(1))\arrow[r,"{\cap (\Gamma|_{G-Y''})}"] 
\arrow[d, phantom, ""{coordinate, name=Z}]
& H_{2\dim G -1}^{BM}(G -Y'', \Z(\dim G-1)) \arrow[dl,
"i_*"',
rounded corners,
to path={ -- ([xshift=2ex]\tikztostart.east)
|- (Z) [near end]\tikztonodes
-| ([xshift=-2ex]\tikztotarget.west)
-- (\tikztotarget)}] \\
&H_{2\dim G -1}^{BM}(C\times X-Y'', \Z(\dim G-1))\arrow[r,equal]&H^{2p+1}(C\times X-Y'', \Z(p)).
\end{tikzcd}
\end{center}
Since the images of the vertical maps are supported on $G$, we have another commutative diagram
\begin{equation}\label{Ad2}
\xymatrixcolsep{1pc}
\xymatrix{
0 \ar[r] & H^{1}(C\times X, \Z(1)) \ar[r] \ar[d]^{\cup\Gamma} & H^{1}(C\times X- Y', \Z(1)) \ar[r]\ar[d]^{(\cup\Gamma)'} &Z^{1}_{Y'}(C\times X)_{\hom}\ar[r]\ar[d]^{\cup \Gamma} & 0\\
0 \ar[r] & N^{p}H^{2p+1}(C\times X, \Z(p+1)) \ar[r] & N^{p}H^{2p+1}(C\times X- Y'', \Z(p+1)) \ar[r] &Z^{p+1}_{Y''}(C\times X)_{\alg}\ar[r] & 0.\\
}
 \end{equation}
Finally, setting $Y'''=\pi_{X}(Y'')$ and letting $j\colon C\times X- \pi_{X}^{-1}(Y''')\rightarrow C\times X-Y''$ be an open immersion,
 we have a commutative diagram
\[
\xymatrixcolsep{1pc}
\xymatrix{
0 \ar[r] & H^{2p+1}(C\times X, \Z(p+1)) \ar[r] \ar[d]^{(\pi_{X})_{*}} & H^{2p+1}(C\times X- Y'', \Z(p+1)) \ar[r]\ar[d]^{(\pi_{X})_{*}j^{*}} &Z^{2p+1}_{Y''}(C\times X)_{\hom}\ar[r]\ar[d]^{(\pi_{X})_{*}} & 0\\
0 \ar[r] & H^{2p-1}(X, \Z(p)) \ar[r] & H^{2p-1}(X- Y''' , \Z(p)) \ar[r] &Z^{p}_{Y'''}(X)_{\hom}\ar[r] & 0,\\
}
 \]
 which restricts to
\begin{equation}\label{Ad3}
\xymatrixcolsep{1pc}
\xymatrix{
0 \ar[r] & N^{p}H^{2p+1}(C\times X, \Z(p+1)) \ar[r] \ar[d]^{(\pi_{X})_{*}} & N^{p}H^{2p+1}(C\times X- Y'', \Z(p+1)) \ar[r]\ar[d]^{(\pi_{X})_{*}j^{*}} &Z^{2p+1}_{Y''}(C\times X)_{\alg}\ar[r]\ar[d]^{(\pi_{X})_{*}} & 0\\
0 \ar[r] & N^{p-1}H^{2p-1}(X, \Z(p)) \ar[r] & N^{p-1}H^{2p-1}(X- Y''' , \Z(p)) \ar[r] &Z^{p}_{Y'''}(X)_{\alg}\ar[r] & 0.\\
}
 \end{equation}
By the diagrams (\ref{Ad1}), (\ref{Ad2}), and (\ref{Ad3}), we have a commutative diagram
\[
\xymatrixcolsep{1pc}
\xymatrix{
0 \ar[r] & H^{1}(C, \Z(1)) \ar[r] \ar[d]^{\Gamma_{*}} & H^{1}(C- Y, \Z(1)) \ar[r]\ar[d]^{(\pi_{X})_{*}j^{*}(\cup \Gamma)' (\pi_{C})^{*}} &Z^{1}_{Y}(C)_{\hom}\ar[r]\ar[d]^{\Gamma_{*}} & 0\\
0 \ar[r] & N^{p-1}H^{2p-1}(X, \Z(p)) \ar[r] & N^{p-1}H^{2p-1}(X- Y''' , \Z(p)) \ar[r] &Z^{p}_{Y'''}(X)_{\alg}\ar[r] & 0.\\
}
 \]
 This is a commutative diagram of mixed Hodge structures.
 The assertion follows by applying $\Ext^{i}_{\MHS}(\Z(0),-)$ and taking the direct limit.
 \end{proof}
 
\begin{cor}\label{Ac1}
The Walker map $\widetilde{\psi}^{p}$ factors through $A^{p}(X)$.
Moreover we have a commutative diagram
 \[
 \xymatrix{
 A^{p}(X)\ar[r]^<<<<<<{\widetilde{\psi}^{p}}\ar@{^{(}->}[d]& J\left(N^{p-1}H^{2p-1}(X,\Z(p))\right)\ar[d]^{\pi^{p}}\\
 CH^{p}(X)_{\hom}\ar[r]^{AJ^{p}} & J^{p}(X).\\
 }
 \]
 \end{cor}
 \begin{proof}
By Lemma \ref{Al1}, we have a commutative diagram
  \[
 \xymatrix{
 \bigoplus_{\Gamma} Z^{1}(\P^{1})_{\hom}\ar[r]^{(AJ^{1})}\ar[d]^{(\Gamma_{*})}& \bigoplus_{\Gamma} J^{1}(\P^{1})=0 \ar[d]^{(\Gamma_{*})}\\
 Z^{p}(X)_{\alg}\ar[r]^>>>>>>{\widetilde{\psi}^{p}} & J\left(N^{p-1}H^{2p-1}(X,\Z(p))\right),\\
 }
 \]
 where $\Gamma$ runs through all codimension $p$ cycles on $\P^{1}\times X$ with the components dominating $\P^{1}$.
Since the image of the left vertical map is the subgroup $Z^{p}(X)_{\rat}\subset Z^{p}(X)$, the first assertion follows. 
The second assertion is immediate by using Lemma \ref{Al0}.
 \end{proof}
 
The source of the Walker map $\widetilde{\psi}^{p}$ will be $A^{p}(X)$ in the following.
 
 \begin{lem}\label{Al2}
 The Walker map $\widetilde{\psi}^{p}$ is functorial for correspondences:
we have a commutative diagram
 \[
 \xymatrix{
 A^{q}(X')\ar[r]^<<<<<<{\widetilde{\psi}^{q}}\ar[d]^{\Gamma_{*}}& J\left(N^{q-1}H^{2q-1}(X',\Z(q))\right)\ar[d]^{\Gamma_{*}}\\
 A^{p}(X)\ar[r]^<<<<<<{\widetilde{\psi}^{p}} & J\left(N^{p-1}H^{2p-1}(X,\Z(p))\right)\\
 }
\]
 for any smooth projective varieties $X, X'$ and codimension $(p-q+\dim X')$ cycle $\Gamma$ on $X'\times X$.
 \end{lem}
 \begin{proof}
Note that the action of $\Gamma$ only depends on its rational equivalence class on $X'\times X$, 
hence we are allowed to use the moving lemma to ensure that $\Gamma$, after moving, comes to intersect properly with finitely many chosen cycles.
Now the result follows from an argument similar to that of Lemma \ref{Al1}, where we may 
always 
assume that 
involved closed subsets have the correct dimensions.
 \end{proof}
 
 \begin{cor}\label{Ac2}
The Walker map $\widetilde{\psi}^{p}$ is surjective. 
Moreover $J(N^{p-1}H^{2p-1}(X,\Z(p)))$ is an abelian variety.
\end{cor}
 \begin{proof}
Let $Z\subset X$ be a closed subset of codimension $p-1$ such that 
the natural map
 \[
 H^{2p-1}_{Z}(X,\Z(p))\rightarrow H^{2p-1}(X,\Z(p))
 \]
induces a surjection
\[
H^{2p-1}_{Z}(X,\Z(p))\rightarrow N^{p-1}H^{2p-1}(X,\Z(p)).
\]
By the right exactness of the intermediate Jacobian functor $J(-)$ \cite{Bei}, we have a surjection
\begin{eqnarray}\label{As1}
J\left(H^{2p-1}_{Z}(X,\Z(p))\right)\rightarrow J\left(N^{p-1}H^{2p-1}(X,\Z(p))\right).
\end{eqnarray}
 Let $\widetilde{Z}$ be a resolution of $Z$ and $\widetilde{Z}_{i}$ be the components of $\widetilde{Z}$.
 An easy computation shows that the natural map
 \[
 \bigoplus_{i} H^{1}(\widetilde{Z}_{i},\Z(1))\rightarrow H^{2p-1}_{Z}(X,\Z(p))
 \]
 is an injection with the cokernel having the trivial Hodge structure.
 This induces a surjection
 \begin{eqnarray}\label{As2}
 \bigoplus_{i} J^{1}(\widetilde{Z}_{i})\rightarrow J\left(H^{2p-1}_{Z}(X,\Z(p))\right).
 \end{eqnarray}
Then we combine (\ref{As1}) and (\ref{As2}) to obtain a surjection
 \[
 \bigoplus_{i} J^{1}(\widetilde{Z}_{i})\rightarrow J\left(N^{p-1}H^{2p-1}(X,\Z(p))\right),
 \]
which coincides with the map induced by the graphs $\Gamma_{i}$ of $\widetilde{Z}_{i}\rightarrow Z \rightarrow X$.
By Lemma \ref{Al2}, we have a commutative diagram
 \[
\xymatrix{
\bigoplus_{i} CH^{1}(\widetilde{Z}_{i})_{\hom} \ar[r]^{AJ^{1}}_{\sim}\ar[d]^{((\Gamma_{i})_{*})}& \bigoplus_{i} J^{1}(\widetilde{Z}_{i}) \ar@{->>}[d]^{((\Gamma_{i})_{*})}\\
 A^{p}(X)\ar[r]^>>>>>>{\widetilde{\psi}^{p}} & J\left(N^{p-1}H^{2p-1}(X,\Z(p))\right).\\
}
 \]
The results follow.
 \end{proof}

 \begin{cor}\label{Ac3}
  The Walker map $\widetilde{\psi}^{p}$ is regular.
 \end{cor}
\begin{proof}
 By Lemma \ref{Al2},  we have a commutative diagram  
 \[
\xymatrix{
CH^{\dim S}(S)_{\hom}\ar[r]^{AJ^{\dim S}}\ar[d]^{\Gamma_{*}}& J^{\dim S}(S)\ar[d]^{\Gamma_{*}}\\
 A^{p}(X)\ar[r]^<<<<<<{\widetilde{\psi}^{p}} & J\left(N^{p-1}H^{2p-1}(X,\Z(p))\right)\\
 }
 \]
 for any smooth projective variety $S$ and codimension $p$ cycle $\Gamma$ on $S\times X$. 
Now the result is immediate using the Albanese map on $S$.
\end{proof}
\begin{rem}
The referee points out that one can slightly modify the definition of a regular homomorphism by allowing $S$ to be any smooth connected (not necessarily projective nor proper) variety.
In fact, the Walker map $\widetilde{\psi}^{p}$ is again regular in this sense: 
the Nagata compactification reduces the assertion to the proper case, for which the same proof as that of Corollary \ref{Ac3} works.
\end{rem}
 
 \begin{proof}[Proof of Theorem \ref{UAJw}]
The Walker map $\widetilde{\psi}^{p}$ constructed in Subsection 2.1 gives a factorization of the Abel--Jacobi map $\psi^{p}$ with desired properties, as shown by Corollaries \ref{Ac1}, \ref{Ac2}, and \ref{Ac3}.
The proof of Theorem \ref{UAJw} is now complete.
\end{proof}

\section{Questions}

\begin{q}
Is the Walker map $\widetilde{\psi}^{p}\colon A^{p}(X)\rightarrow J\left(N^{p-1}H^{2p-1}(X,\Z(p))\right)$ always universal among all regular homomorphisms?
\end{q}

This is related to another question of Murre:
\begin{q}[\protect{\cite[p.~132]{GMV}}]
Does there always exist a universal regular homomorphism?
\end{q}

\end{document}